\documentclass[a4paper,12pt]{amsart}

\usepackage{amsmath}	

\allowdisplaybreaks[1]

\theoremstyle{break}
\newtheorem{thm}{Theorem}[section]
\newtheorem{lemma}[thm]{Lemma}
\newtheorem{prop}[thm]{Proposition}
\newtheorem{cor}[thm]{Corollary}

\theoremstyle{remark}

\newtheorem{example}[thm]{Example}

\theoremstyle{definition}
\newtheorem{definition}[thm]{Definition}

\newcommand{\R}{\Bbb{R}}

\newcommand{\N}{\Bbb{N}}
\newcommand{\Z}{\Bbb{Z}}

\newcommand{\supp}{\mathop\mathrm{supp}}

\newcommand{\conv}{\mathop\mathrm{conv}}
\newcommand{\bconv}{\mathop\mathrm{bconv}}
\newcommand{\rint}{\mathop\mathrm{rint}}

\newcommand{\LT}{\mathop\mathrm{LT}}

\newcommand{\tf}{\tilde{f}}
\newcommand{\tr}{\tilde{r}}
\newcommand{\br}{\hat{r}}

\newcommand{\talpha}{\tilde{\alpha}}
\newcommand{\tM}{\widetilde{M}}

\newcommand{\tGamma}{\widetilde{\Gamma}}
\newcommand{\tgamma}{\widetilde{\gamma}}

\renewcommand{\epsilon}{\varepsilon}
\renewcommand{\phi}{\varphi}

\title[Notes on Newton diagrams and sums of squares]{Notes on optimality conditions using Newton diagrams and sums of squares}
\author{Yoshiyuki Sekiguchi}
\address{Graduate School of Marine Science and Technology, Tokyo
University of Marine Science and Technology, Etchujima 2-1-8, Koto,
Tokyo 135-8533, Japan}
\email{yoshi-s@kaiyodai.ac.jp}

\begin{document}

\maketitle

\begin{abstract}
 We consider relationships between optimality conditions using Newton
 diagrams and sums of squares of polynomials and power series.
\end{abstract}

\section{Introduction}

We consider the set of sums of squares of real polynomials $\R[x]$
denoted by
$\sum\R[x]^2$ 
and the quadratic module 
$M(g_1,\ldots,g_l)=\{\sum_i\sigma_ig_i\mid \sigma_i\in \sum\R[x]^2\}$
generated by $g_i\in \R[x], i = 1,\ldots, l$.
In addition, let sums of squares of power series $\R[[x]]$ 
be denoted by $\sum\R[[x]]^2$
and $\widetilde{M}(g_1,\ldots,g_l)=\{\sum_i \tau_i g_i\mid \tau_i\in \sum\R[[x]]^2\}$.
It is well known that these play important roles in polynomial
optimization problems; see \cite{Mar} and references therein.
On the other hand, optimality conditions in optimization theory
can be used to give sufficient conditions for a function to belong to quadratic
modules
generated by constraint functions (sos-representability).

A polynomial optimization problem is the following:
\begin{align*}
   \mathrm{(POP)}\ & \min\quad f(x)\\
 &  \text{ s.t.}\quad  g_i(x)\geq 0, i = 1, \ldots l, \\
 &  \phantom{\text{ s.t.}\quad} h_j(x)=0, j = 1, \ldots, m,
\end{align*}
where $f,g_i,h_j\in\R[x]=\R[x_1,\ldots,x_n]$.
We say the second order condition holds at $z$
if $z$ is a minimizer and there exist $\lambda_i\geq 0,
\mu_i\in \R$ such that
$\nabla f(z) = \sum_i\lambda_i\nabla g_i(z) + \sum_j\mu_j\nabla h_j(z)$, 
$\lambda_i g_i(z) = 0$
and 
\[
 \nabla^2\left(f - \sum_i\lambda_i g_i - \sum_j \mu_j h_j \right)(z)
\]
is positive definite on the subspace 
$\{x\in \R^n \mid \lambda_i\nabla g_i(z)x=0, \nabla h_j(z)x = 0\}$.
Then \cite{BSch}, \cite{N} showed that 
if the second order condition and some constraint qualification conditions hold
at each global minimizer, then 
$f - f_{\mathrm{min}}$ is contained in the quadratic module $M(g_1, \ldots, g_l) + \langle
h_1,\ldots, h_m \rangle$, where $f_{\mathrm{min}}$ is the global minimum.


We are interested in relationships between other optimality conditions
and sos-representability.
In this notes, we investigate an optimality condition
using Newton diagrams given in \cite{V}.


In \cite[Theorem 4.12]{Y},
the author had imposed the condition that
\begin{center}
$(*)$\quad for each maximal face $\gamma$ of $\Gamma(f)$,
 $f_\gamma(x)\in \rint\left(\sum\R[x]^2_{\frac{1}{2}\gamma}\right)$.
\end{center}
However Condition $(*)$ is insufficient for the theorem to hold.
The stronger condition
\[
 (**)\quad f_\Gamma(f)\in
\rint\left(\sum\R[x]^2_{\frac{1}{2}\gamma}\right),
\]
was essentially used in the last sentence of the proof.
In fact, Christoph Schulze gave the counterexample
\[
 f(x,y,z) = x^2y^2z^2+2x^8+3y^{10}+4z^{14}+4x^4z^7+6y^5z^7,
\]
which fulfills all conditions of \cite[Theorem 4.12]{Y} but is not a sum of
squares.
We are grateful to him for reporting and for interesting discussions.
To fix it, we provide a corrected version in Section 4.
The results before the proof of Theorem 4.12 of \cite{Y} remain to hold
without any change.
We correct Theorem 4.17, 4.21, Corollary 4.22 and Theorem 5.3 in
\cite{Y}
by replacing Condition $(*)$ with Condition $(**)$.

\section{Preliminaries}

For a polyhedral convex set $P\subset \R^n$, $F\subset P$ is called a \textit{face} of $P$, 
if there exists a supporting hyperplane $H$ such that $F = P\cap H$.

For $f\in \R[x]$, the \textit{support} of $f$ is the set of all exponents of
monomials of $f$ and be denoted by $\supp f$.
For $\alpha\in \Z_+^n$, $|\alpha|=a_1 + \cdots + a_n$ and $\alpha$ is 
said to be \textit{even} if all coordinates are even.
Let
\begin{align*}
 \Delta(f) &= \bigcup\{\alpha + \R_+^n \mid \alpha \in \supp f\},\\
 \Delta_E(f)&=\bigcup\{\alpha + \R_+^n \mid \alpha \in  \supp
 f\cap (2\Z)^n\}. 
\end{align*}
The convex hull $\conv \Delta(f)$ of $\Delta(f)$ 
is called the \textit{Newton polyhedron} of $f$.
       The \textit{Newton diagram} $\Gamma(f)$ is the union of the
       compact faces of $\conv \Delta(f)$.
For $\gamma \subset \R_+^n$, define 
$f_\gamma = \sum\{f_\alpha x^\alpha\mid \alpha \in \gamma \cap \supp f\}$ 
and $\R[x]_\gamma$ as the set of
polynomials whose  supports are included in 
$\gamma \cap \Z^n$.
The polynomial $p_\gamma=\sum_{\alpha\in \gamma \cap (2\Z)^n}x^\alpha$ is called
the \textit{principal polynomial} of $\gamma$.

We consider the \textit{finest locally convex topology} on $\R[x]$; see
\cite{CMN}, \cite{Schea}.
This topology is Hausdorff and each finite dimensional subspaces of $\R[x]$
inherits the Euclidean topology, 
and every converging sequence in $\R[x]$ is contained in a finite
dimensional subspace.
For a subset $C$ of a finite dimensional subspace of $\R[x]$,
the \textit{relative interior} $\rint C$ is defined as 
the interior of $C$ with respect to the minimal finite dimensional 
subspace which includes $C$.

\section{Necessary condition}

Vasil'ev showed a necessary condition 
for locally isolated minimality using Newton diagrams \cite[Theorem 1.5 (1)]{V}.
\begin{thm}
[Vasil'ev]
 Let $f\in \R[x]$ with $f(0)=0$ have an isolated minimum at $0$. 
Then 
\begin{enumerate}
 \item $\Gamma(f)$ meets all coordinate axes;
 \item Every vertex of $\Gamma(f)$ is even;
 \item For each vertex $\alpha$ of $\Gamma(f)$, $f_\alpha>0$;
 \item For each face $\gamma$ of $\Gamma(f)$, $f_\gamma(x)\geq0,\
       \forall x\in \R^n$.
\end{enumerate}

\end{thm}

The following theorem gives necessary conditions using Newton diagrams for sos-representability.
\begin{thm}
\label{thm:necessity}
  Let $f\in \R[x]$ with $f(0)=0$ be a sum
 of square 
 polynomials. Then
\begin{enumerate}
 \item Every vertex of $\Gamma(f)$ is even.
 \item For each vertex $\alpha$ of $\Gamma(f)$, $f_\alpha>0$.
 \item For each face $\gamma$ of $\Gamma(f)$, $f_\gamma(x)\in\sum\R[x]^2$.
\end{enumerate}
\end{thm}
\begin{proof}
As a easy consequence of the proof of  \cite[Proposition 1.2]{V}, we
 have that nonnegativity of $f$ implies the properties $(1)$ and $(2)$

We will show the properties $(3)$.

For each face $\gamma \subset \Gamma(f)$, 
let $\Delta =\{\alpha \in \Z_+^n\mid A_1\alpha_1 + A_2\alpha_2 + \cdots + A_n\alpha_n = v\}$
be the supporting hyperplane including the face $\gamma$
 but not $\Gamma(f)\setminus \gamma$.
Here we may assume $A = (A_1,\ldots, A_n)\in \Z_+^n\setminus\{0\}^n$
and hence $v = \min\{A\cdot\alpha \mid
 \alpha \in \supp f\}$, where the dot product is defined by $A\cdot\alpha=\sum_{i}A_i\alpha_i$.
We can write $f= f_v + f_{v+1} + \cdots$, where $f_{v'}$ is a polynomial 
each of whose exponents $\alpha$ satisfy
 $A\cdot\alpha = v'$. Then we have $f_v = f_\gamma$.

Next, let $f = \sum_i^s g_i^2$. We define $w_i = \min\{A\cdot\alpha \mid 
\alpha \in \supp g_i \}$, $w = \min\{w_1, \ldots w_s\}$.
For $i = 1,\ldots, s$, $g_i$ is decomposed as $g_i = g_{i, w} + g_{i, w+1} + \cdots + g_{i, w + t}$ for some $t_i\in \N$.
 Then we write
 \[
 f = \sum_i^s\left(g_{i,w} + g_{i,w+1} + \ldots + g_{i,w + t_i}\right)^2
 =\sum_i^s g_{i,w}^2 + \tf,
 \]
 where all exponents $\alpha$ of $\tf$ 
satisfies $A\cdot\alpha > 2w$.
 Since $\sum_{i=1}^s g_{i,w}^2 \neq 0$, we have $v \leq 2w$.
 If $v < 2w$, there exists $\beta \in \gamma$ such that
 $A\cdot\beta = v < 2w$ and 
 $x^\beta$ is a monomial of $f - \sum_i^s g_{i,w}^2 = \tf$. This is a
 contradiction and we have $v = 2w$.
Therefore $f_\gamma = \sum_i^s g_{i,w}^2$.

\end{proof}

\section{Sufficient condition}

We investigate sufficient conditions for a polynomial to be a sum
of squares of power series $\R[[x]]$ using Newton diagrams.
We present a sufficient condition for locally isolated
minimality by Vasil'ev \cite[Theorem 1.5 (2)]{V}.

\begin{thm}
 [Vasil'ev]
 \label{thm:V}
 Let $f\in \R[x]$ with $f(0)=0$.
\begin{enumerate}
 \item $\Gamma(f)$ meets all coordinate axes.
 \item Every vertex of $\Gamma(f)$ is even.
 \item For each vertex $\alpha$ of $\Gamma(f)$, $f_\alpha>0$.
 \item For each face $\gamma$ of $\Gamma(f)$, $f_\gamma(x)>0, \forall x
       \text{ with }  x_1\cdots x_n\neq 0$.
\end{enumerate}
Then $f$ has an isolated minimum at $0$.
 \end{thm}

\begin{example}
 \[
 f(x,y) = x^6 + x^4y + x^3y^3 + x^2y^2 + y^4
\]
The vertices of the Newton diagram of $f$ are $(0,4)$, $(2,2)$ and $(6,0)$. 
The compact faces consist of $\gamma_1=\{t(0,4)+ (1-t)(2,2)\mid 0\leq t \leq
 1\}$, $\gamma_2=\{t(2,2)+ (1-t)(6,0)\mid 0\leq t \leq
 1\}$ and the vertices. Here we have for $x,y$ with $xy\neq 0$, 
\begin{align*}
 f_{\gamma_1} & = x^2y^2 + y^4>0\\
 f_{\gamma_2} & = x^6 + x^4y + x^2y^2 = x^2\left\{\left(x^2 +
 \frac{1}{2}y\right)^2 + \frac{3}{4}y^2\right\}>0.
\end{align*}
Therefore $(0,0)$ is an isolated minimum of $f$.
\end{example}

\subsection{Simple Newton diagrams}
We seek conditions which are analogous to the one by Vasil'ev.
We consider the following well-known sufficient condition
from the point of view of Newton diagrams; see e.g.\ \cite[Lemma 9.5.1]{Mar}.
 \begin{lemma}
  \label{lemma:pd}
 Let $f\in \R[x]$.
 Suppose $f = \sum_k f_k$ be the expansion of
 its homogeneous components where $\deg f_k = k$.
 If $f_0=f_1 = 0$ and $f_2$ is a positive definite form,
 then $f\in \sum\R[[x]]^2$.
 \end{lemma}
Here we note that if $f_2$ is positive definite, then $f_2\in
\rint(\sum\R[x]_1^2)$ \cite[Corollary 2.5, Remark 2.6]{GM}.
Thus the lemma tells us that $f\in \sum\R[[x]]^2$ if the Newton diagram $\Gamma:=\Gamma(f)$ is 
contained in the plane $|\alpha| =2$ and $f_\Gamma$ is
contained in $\rint(\sum\R[x]_1^2)$.
From this observation, we first obtain an extension of the lemma in the case that the Newton
diagram
is contained in a plane which is parallel to $|\alpha| =2$.
\begin{thm}
 \label{thm:homogeneous}
 Let $f_{2m}$ be the lowest homogeneous part of $f\in \R[x]$.
 If $f_{2m}\in \rint\left(\sum\R[x]_m^2\right)$, then $f\in \sum\R[[x]]^2$.
\end{thm}
To show this, we need the following lemmas.
In addition, we will use the well-known fact that for any $u\in \R[x]$ with $u(0) = 0$, 
\[
 1 + u,\  \frac{1}{1 + u}\in \sum\R[[x]]^2,
\]
see e.g. \cite[Section 1.6]{Mar}.

 \begin{lemma}
  \label{lemma:GM}
 Suppose $f\in \R[x]$ is a homogeneous polynomial of degree $2d$
 and $\{e_i\}$ is the canonical basis of $\Z^n$.
 Then there exists $\tM>0$ such that 
 $f + \sum_{i=1}^n M x^{2de_i}\in \sum\R[x]^2$ for $M>\tM$.
 \end{lemma}
 \begin{proof}
  This is easily implied
by Ghasemi-Marshall \cite[Theorem 2.1]{GM}.
 \end{proof}
\begin{lemma}
\label{lemma:principal}
Let $f\in \R[x]$ and $\gamma$ be a face of $\Gamma(f)$. 
Then we have the following:
\begin{enumerate}
\item  The principal polynomial $p_\gamma$
 of $\gamma$ lies in
	       $\rint\left(\sum\R[x]^2_{\frac{1}{2}\gamma} \right)$.

\item $f_\gamma \in \rint\left(\sum\R[x]^2_{\frac{1}{2}\gamma} \right)$
	       if and only if $f_\gamma - \epsilon p_\gamma\in 
	       \sum\R[x]_{\frac{1}{2}\gamma}^2$ for sufficiently small $\epsilon>0$.
\end{enumerate} 
\end{lemma}
\begin{proof}
The proof of $(1)$ is almost identical to the one given in \cite[Proposition 5.5]{CLR}.
 Let $g\in \sum\R[x]^2_{\frac{1}{2}\gamma}$.
Then $g = \sum_t h_t^2$ for some $h_t\in \R[x]_{\frac{1}{2}\gamma}$.
For each $a\in \supp g$, there exist $b_1,b_2\in \bigcup_t \supp h_t$ such that
$a = b_1 + b_2$.
Since we have $b_1, b_2\in \frac{1}{2}\gamma$, 
there exist $\alpha_1, \alpha_2\in \gamma\cap (2\Z)^n$ such that
$a = \frac{1}{2}(\alpha_1 + \alpha_2)$.
Since
\[
 x^{\alpha_1} + x^{\alpha_2} \pm 2x^a = \left(x^{\frac{1}{2}\alpha_1} 
\pm x^{\frac{1}{2}\alpha_2}\right)^2,
\]
we conclude that $p_\gamma \pm 2x^a\in \sum\R[x]^2_{\frac{1}{2}\gamma}$
and hence that $p_\alpha - \epsilon
 g\in \sum\R[x]^2_{\frac{1}{2}\gamma}$ for sufficiently small $\epsilon>0$.

For $(2)$, consider $V$ as the affine hull of
 $\sum\R[x]^2_{\frac{1}{2}\gamma}$ in the proof of \cite[Proposition 1.4]{CMN}.

\end{proof}

 \begin{proof}
  [Proof of Theorem $\ref{thm:homogeneous}$]
  Let $\Gamma=\Gamma(f_{2m})$. Since $f_{2m}\in
  \rint\left(\sum\R[x]_m^2\right)$,
  $\Gamma$ meets all coordinate axes and then $\Gamma=\Gamma(f)$.
 By $(2)$ of Lemma $\ref{lemma:principal}$, there exists $\epsilon>0$
 such that $f_{2m} - 2\epsilon p_\Gamma\in \sum\R[x]^2$.
 Let $t = \lceil\frac{1}{2}\deg f \rceil$ and $\{e_i\}$ be the canonical
 basis of $\Z^n$. Then we write
 \[
  f_{2m} = f_{2m} - 2\epsilon p_\Gamma + f^{(1)} + f^{(2)},
 \]
 where for $M_k>0$,
\begin{align*}
f^{(1)} & = \epsilon p_\Gamma
 -\sum_{k=m+1}^t\sum_{i=1}^n M_{2k} x^{2ke_i},\\
f^{(2)} & = \epsilon p_\Gamma + \sum_{k=m+1}^t\sum_{i=1}^n M_{2k} x^{2ke_i} + \sum_{|\alpha|\geq
 2m+1}f_\alpha x^\alpha.
\end{align*}
 Since $\Gamma$ meets all coordinate axes, $p_\Gamma$ contains $x^{2m
 e_i}$
 for all $i$. Then we have
 \[
 \epsilon x^{2me_i} - \sum_{k=m+1}^t M_{2k} x^{2ke_i}
 = x^{2me_i}\left(\epsilon - \sum_{k=m+1}^t M_{2k} x^{(2k-2m)e_i} \right)
 \in \sum\R[[x]]^2
 \]
 and hence $f^{(1)}\in \sum\R[[x]]^2$ for any $M_k>0$.
 
 Next we will show $f^{(2)}\in \sum\R[x]^2$.
 We claim that for arbitrary $C_\alpha>0$, there exists $D>0$
 such that
 \[
  T_k:=\sum_{\alpha:\text{even}\atop{|\alpha|=2k}} C_\alpha x^{\alpha}
 + \sum_{|\alpha|=2k+1}f_\alpha x^\alpha +
 \sum_{i=1}^n D x^{(2k+2)e_i}
 \]
 is contained in $\sum\R[x]^2$.
 Let $\alpha\in \Z_+^n$ with $|\alpha|=2k+1$.
For the index $s$ such that
 $\sum_{i}^s 2\alpha_i \leq 2k < \sum_{i}^{s+1} 2\alpha_i$,
 we define $\beta(\alpha),\beta'(\alpha)\in \Z_+^n$ as
 \[
  \beta(\alpha)_i=\begin{cases}
	  2\alpha_i,\quad i=1,\ldots, s\\
	  2k - \sum_i^s 2\alpha_i,\quad i = s+1\\
	  0, \quad \text{otherwise}
	  \end{cases}
 \]
 and $\beta'(\alpha) = 2\alpha - \beta(\alpha)$. Then $\beta(\alpha), \beta'(\alpha)$ are even, $|\beta(\alpha)|=2k, |\beta'(\alpha)|=2k+2$ and
 $2\alpha = \beta(\alpha) + \beta'(\alpha)$. Thus
 \[
 C_{\beta(\alpha)} x^{\beta(\alpha)} + f_\alpha x^\alpha
 = C_{\beta(\alpha)}\left(x^{\frac{\beta(\alpha)}{2}} +
 \frac{f_\alpha}{2C_{\beta(\alpha)}}x^{\frac{\beta'(\alpha)}{2}}\right)^2
 - \frac{f_\alpha^2}{4C_{\beta(\alpha)}}x^{\beta'(\alpha)}.
 \]
 Let
 \begin{align*}
  & S(k) = \left\{\alpha\in \supp f \mid |\alpha| = k \right\},\\
   & I = \{\beta\in \Z_+^n\mid \beta\text{ is even}, |\beta|=2k\}
  \setminus \left\{\beta(\alpha) \mid S(2k+1) \right\}.
 \end{align*} 
 Then we have
\begin{align*}
  T_k 
 & =  \sum_{\beta\in I}C_\alpha x^\alpha
 + \sum_{\alpha\in S(2k+1)}
 \left(C_{\beta(\alpha)}x^{\beta(\alpha)} + f_\alpha x^\alpha \right)
 +\sum_{i=1}^n D x^{(2k+2)e_i} 
\\
 & = \sum_{\beta\in I}C_\alpha x^\alpha
 + \sum_{\alpha\in S(2k+1))}
 C_{\beta(\alpha)}\left(x^{\frac{\beta(\alpha)}{2}} +
 \frac{f_\alpha}{2C_{\beta(\alpha)}}x^{\frac{\beta'(\alpha)}{2}}
 \right)^2 \\
& \hspace{15ex} + \left(\sum_{i=1}^n D x^{(2k+2)e_i} - \sum_{|\alpha|=2k+1 \atop{\alpha\in
 \supp f}}\frac{f_\alpha^2}{4C_{\beta(\alpha)}}x^{\beta'(\alpha)}\right).
\end{align*}
 Here the last parenthesis is a homogeneous polynomial of degree $2k+2$ 
 and Lemma $\ref{lemma:GM}$ implies that
 it is a sum of square polynomials for sufficiently large $D_i$.
 Thus the claim is proved.

 Now we have
\begin{align*}
  f^{(2)} & = \epsilon p_\Gamma
 + \sum_{k=m+1}^{t}\sum_{i=1}^n M_{2k} x^{2ke_i}
 + \sum_{k=m+1}^{t}\left(\sum_{|\alpha|=2k-1}f_\alpha x^\alpha
 + \sum_{|\alpha|=2k}f_\alpha x^\alpha\right)\\
 & = g^{(1)} + g^{(2)} + g^{(3)},
 \end{align*}
where
 \begin{align*}
 g^{(1)} & = \epsilon p_\Gamma + \sum_{|\alpha| = 2m+1}f_\alpha x^\alpha
 + \sum_{i=1}^n \frac{M_{2m+2}}{4} x^{(2m+2) e_i}\\
   g^{(2)} & = \sum_{k=m+2}^{t}\left(
 \sum_{\alpha\in S(2k-1)}x^{\beta(\alpha)}
 + \sum_{|\alpha|=2k-1}f_\alpha x^\alpha
 + \sum_{i=1}^n \frac{M_{2k}}{4} x^{2ke_i}
 \right) \\
  g^{(3)} & = \sum_{k=m+2}^{t}\left(
 \sum_{i=1}^n \frac{M_{2k-2}}{4} x^{(2k-2)e_i}
 - \sum_{\alpha\in S(2k-1)}x^{\beta(\alpha)}
 \right) \\
 g^{(4)} & = \sum_{k=m+1}^{t}\left(
 \sum_{|\alpha|=2k}f_\alpha x^\alpha
+ \sum_{i=1}^n \frac{M_{2k}}{2}x^{2ke_i}
 \right) + \sum_{i=1}^n \frac{M_{2t}}{4}x^{2te_i}.
\end{align*}
 Note that $\sum_{\alpha \in S(2k-1)} x^{\beta(\alpha)}$ is a homogeneous
 polynomials of degree $2k-2$.
 Again by Lemma $\ref{lemma:GM}$,
 there exist $\tM$ such that $g^{(3)},g^{(4)}\in \sum\R[x]^2$
 for $M_{2k}>\tM$.
 The claim above implies that there exist $M_{2m+2}>\tM$ such that
 $g^{(1)}\in \sum\R[x]^2$. Similarly for $k = m+2, \ldots t$,
 there exist $M_{2k}>\tM$ such that $g^{(2)}\in \sum\R[x]^2$.
 Therefore $f^{(2)}\in \sum\R[x]^2$ and hence $f\in \sum\R[[x]]^2$. 
 \end{proof}

\begin{example}
 Consider
 \[
  f(x,y,z) = 2x^6 + 2y^6 + 2z^6 + xy^3z^3 + x^2y^4z^3.
 \]
 The lowest homogeneous part is $2x^6+2y^6 +2z^6$, which is
 contained in $\rint(\sum\R[x]_2^2)$.
 The monomials $xy^3z^3$ and $x^2y^4z^3$ are not even and their exponent
 vectors are $(1,3,3)$ and $(2,4,3)$ respectively.
 Now we have
 \begin{align*}
  2(1,3,3) & = (2,6,6) = (2,4,0) + (0,2,6)\\
  2(2,4,3) & = (4,8,6) = (4,4,0) + (0,4,6)
 \end{align*}
 and then
\begin{align*}
 x^2y^4 + xy^3z^3 & =  \left(xy^2 + \frac{1}{2}yz^3\right)^2
 - \frac{1}{4}y^2z^6 \\
 x^4y^4 + x^2y^4z^3 & = \left(x^2y^2 + \frac{1}{2}y^2z^3\right)^2
 - \frac{1}{4}y^4z^6.
\end{align*}
 Now we have
 \begin{align*}
  f & = x^6 + y^6 + z^6 - 2a(x^8 + y^8 +z^8) - b(x^{10} + y^{10} +
  z^{10})\\
    & \hspace{5ex}+ (x^6 + y^6 + z^6 - \epsilon x^2y^4)\\
  & \hspace{5ex}+[\epsilon x^2y^4 + xy^3z^3 
  + a(x^8 + y^8 + z^8)] \\
  & \hspace{5ex }+ [x^4y^4 + x^2y^4z^3 + b(x^{10} +
  y^{10} + z^{10})]\\
& \hspace{5ex}  + [a(x^8 + y^8 + z^8) - x^4y^4]\\
 &  = x^6(1 - 2ax^2 - bx^4) + y^6(1 - 2ay^2 - by^4)
  + z^6(1 - 2az^2 - bz^4)\\
    & \hspace{5ex}+ (x^6 + y^6 + z^6 - \epsilon x^2y^4)\\
 & \hspace{5ex} + \left[\epsilon\left(xy^2 +
  \frac{1}{2\epsilon}yz^3\right)^2
  -\frac{1}{4\epsilon}y^2z^6 + a(x^8 + y^8 + z^8)
  \right]\\
  & \hspace{5ex}
  + \left[\left(x^2y^2 + \frac{1}{2}y^2z^3\right)^2
  - \frac{1}{4}y^4z^6 + b(x^{10} +  y^{10} + z^{10})\right]\\
& \hspace{5ex}  + \left[a(x^8 + y^8 + z^8) - x^4y^4\right]
 \end{align*}
 By Lemma $\ref{lemma:principal}$, there exists $\epsilon>0$ such that
 $x^6 + y^6 + z^6 - \epsilon x^2y^4\in \sum\R[x]^2$.
 Then by Lemma $\ref{lemma:GM}$, we can choose $a, b>0$ large enough so
 that the last three brackets are contained in $\sum\R[x]^2$.
 Therefore $f\in \sum\R[[x]]^2$.
\end{example}


\subsection{General Newton diagrams}
Next, we consider the case that the Newton diagram has several faces
which are contained in different planes.
For this general case, we need an assumption
on the distributions of exponent vectors of polynomials
in addition to conditions corresponding to those of Theorem $\ref{thm:V}$.

For $\alpha^1,\ldots,\alpha^t\in (2\Z_+)^n$, a \textit{binary convex combination}
of these points
is $\alpha\in \Z_+^n$ which can be written as
\[
 \alpha = \lambda_1 \alpha^1 + \cdots + \lambda_t \alpha^t,
\]
for some $\lambda_s>0, \sum_{s=1}^t\lambda_s = 1$
such that $2$-adic expansions of $\lambda_1,\ldots,\lambda_t$ have finite digits.
We also say that a binary convex combination has \textit{full digits}
if there exists $N\in \N$ such that
\begin{enumerate}
 \item $\lambda_s = \sum_{k=1}^N \delta_{sk}2^{-k}$ for $\delta_{sk}\in
\{0,1\}$, $s = 1,\dots,t$;
\item for each $k$, there exists $s$ with $\delta_{sk}=1$.
\end{enumerate}

For $\Delta_E\subset \Z_+^n$,
the set of all binary convex combinations of points in $\Delta_E\cap (2\Z_+)^n$
 which have full digits and are contained in $\Z^n$ is called the \textit{bisectional convex
       hull}
       of $\Delta_E$ and denoted by $\bconv\Delta_E$.
 Note that
 we have
 \[
 \Delta_E\cap Z^n \subset \bconv\Delta_E \subset \conv \Delta_E\cap Z^n. 
 \]
\begin{example}
 \label{ex:binary}
 Let $\Delta_E = \left\{(16,0) + \Z_+^2\right\}
 \cup \left\{(0,10) + \Z_+^2\right\}$. Then
 $(11,7)\in \bconv \Delta_E$. In fact, we have
\begin{align*}
 (11,7) & = \left(\frac{1}{2} + \frac{1}{2^3}\right)(16,0)
 + \frac{1}{2^2}(4,22) + \frac{1}{2^3}(0,12),
\end{align*}
 $(4,22), (0,12) \in \Delta_E\cap(2\Z_+)^n$ and it has full digits.
\end{example}
  \begin{prop}
   \label{prop:normal_binary}
   Let $\Delta_E\subset \Z_+^n$. Then we have
\begin{multline*}
    \bconv\Delta_E = \Z^n \cap \\
   \left\{\sum_{k=1}^N 2^{-k}\beta^k +
   2^{-N}\beta^{N+1}\mid \beta^k\in \Delta_E\cap(2\Z_+)^n,
   k=1,\ldots,N+1 \text{ for some } N\in \N\right\}.
\end{multline*}  
  \end{prop}
 \begin{proof}
  Let $\alpha \in \bconv\Delta_E$. Then there exist
  $\alpha^1,\ldots,\alpha^t\in \Delta_E\cap (2\Z_+)^n,
  \lambda_1,\ldots,\lambda_t>0, \sum_{s=1}^t\lambda_s=1$ such that
  \[
   \alpha = \lambda_1\alpha^1 + \cdots + \lambda_t\alpha^t
  \]
  and it has full digits.
  Suppose that $\lambda_s = \sum_{k=1}^{N+1}\delta_{sk}2^{-k}$ for
  $s=1,\ldots, t$.
  Since $\sum_{s=1}^t\lambda_s = 1$ and $\{\delta_{sN+1}\}_s$ corresponds
  to the $N+1$ st digits which are the last ones,
  the number of nonzero $\{\delta_{sN+1}\}_s$ is even.
  Thus there exist at least two nonzero $\delta_{s'N+1},\delta_{s''N+1}$.
  
  Since $\alpha$ has full digits, for each $k=1,\ldots, N$, there exists $\tau\in
  \{1,\ldots,t\}$
  such that $\delta_{\tau k}=1$ and then let $\tau(k)$ be the least such
  index.
  Then we have
\begin{multline*}
 1 = \sum_{s=1}^t\lambda_s \geq  \sum_{k=1}^{N}\delta_{\tau(k)k}2^{-k} + \delta_{s'N+1} 2^{-N-1}
 + \delta_{s''N+1} 2^{-N-1} \\
 = \sum_{k=1}^{N}2^{-k} + 2^{-N-1} + 2^{-N-1}=1.
\end{multline*}
Therefore there is only one $s$ with $\delta_{sk}=1$ for each
  $k=1,\ldots,N$.
  It gives the desired representation.
 \end{proof}
 
  \begin{prop}
   \label{prop:integrity_binary}
  For $\beta^k\in (2\Z_+)^n, k =1,\ldots N+1$, let  
  \[
   \alpha = \sum_{k=1}^N \frac{1}{2^k}\beta^k + \frac{1}{2^N}\beta^{N+1}.
  \]
  be contained in $\Z^n$.
  Then we have
  \[
   \sum_{k=N'}^{N}\frac{1}{2^{k-N'+2}}\beta^k + \frac{1}{2^{N-N'+2}}\beta^{N+1}.
  \]
  is contained in $\Z_+^n$ for $N'=2,\ldots, N+1$ with the convention
  $\sum_{k=N+1}^N a_k= 0$.
  \end{prop}
  \begin{proof}
   Since $\beta^k\in (2\Z_+)^n$, the left hand side of
  \[
  2^{N'-2}\left(\alpha - \sum_{k=1}^{N'-1}\frac{1}{2^k}\beta^k\right)
  = \sum_{k=N'}^{N}\frac{1}{2^{k-N'+2}}\beta^k + \frac{1}{2^{N-N'+2}}\beta^{N+1}.
  \]
   is contained in $\Z_+^n$ and so is the right hand side.
  \end{proof}
\begin{example}
By Example $\ref{ex:binary}$,
\[
   (11,7) = \frac{1}{2}(16,0) + \frac{1}{2^2}(4,22)
 + \frac{1}{2^3}(16,0) + \frac{1}{2^3}(0,12) 
\]
and $(4,22)\in \Delta_E$.
In addition, we have all of right hand sides of
\begin{align*}
&   (11,7) - \frac{1}{2}(16,0)  = \frac{1}{2^2}(4,22)
 + \frac{1}{2^3}(16,0) + \frac{1}{2^3}(0,12),\\
&   2\left((11,7) - \frac{1}{2}(16,0) - \frac{1}{2^2}(4,22)\right) = 
 \frac{1}{2^2}(16,0) + \frac{1}{2^2}(0,12)\\
&   2^2\left((11,7) - \frac{1}{2}(16,0) - \frac{1}{2^2}(4,22) - \frac{1}{2^3}(16,0)\right) = 
 \frac{1}{2}(0,12)
\end{align*}
are contained in $\Z_+^2$.

\end{example}
  Now we present sufficient conditions.
 \begin{thm}
\label{thm:sufficiency}
  Let $f\in \R[x]$ with $f(0)=0$. Suppose that
\begin{enumerate}
 \item Every vertex of $\Gamma(f)$ is even.
 \item For each vertex $\alpha$ of $\Gamma(f)$, $f_\alpha>0$.
 \item $f_{\Gamma(f)}(x)\in \rint
       \left(\sum\R[x]_{\frac{1}{2}{\Gamma(f)}}^2\right)$.
\item If for each maximal face $\gamma$ of $\Gamma(f)$,
        \begin{multline*}
         \{\alpha \in \supp f\cap \conv\Delta(f_\gamma)\setminus \gamma\mid \alpha \text{ is odd or } f_\alpha<0\}
	 \subset \bconv \Delta_E(f_\gamma).
       \end{multline*}
\end{enumerate}
Then $f\in \sum\R[[x]]^2$.
 \end{thm}

We note that by Theorem $\ref{thm:necessity}$, Condition $(3)$ of Theorem $\ref{thm:sufficiency}$
implies the corresponding interiority condition for each face of $\Gamma(f)$.
To show the theorem, we need the following lemmas.

   \begin{lemma}
    \label{lemma:binary_sos}
    Let
  \[
   \alpha = \sum_{k=1}^N \frac{1}{2^k}\beta^k + \frac{1}{2^N}\beta^{N+1},
  \]
  where $\beta^k\in (2\Z_+)^n, N\in \N$.
   For any $\epsilon>0, a\in \R,t\in \{1,\ldots,N+1\}$
   there exists $M>0$ such that
\[
\sum_{k=1}^{N+1}\epsilon x^{\beta^k} -  ax^\alpha + M x^{\beta^{t}} \in \sum\R[x]^2. 
\]
   \end{lemma}
   \begin{proof}
   Case $t = N+1$.
\begin{align*}
&     \sum_{k=1}^{N+1} \epsilon x^{\beta^k} - ax^\alpha 
    \\
 & = \sum_{k=2}^{N+1}\epsilon x^{\beta^k}
 + \epsilon\left(x^{2^{-1}\beta^1}
 - \frac{a}{2\epsilon}x^{\sum_{k=2}^N 2^{-k}\beta^k +
 2^{-N}\beta^{N+1}}\right)^2  \\
& \hspace{5ex} - \epsilon\left(\frac{a}{2\epsilon}\right)^2 x^{\sum_{k=2}^N 2^{-k+1}\beta^k + 2^{-N+1}\beta^{N+1}}\\
 & = \sum_{k=3}^{N+1}\epsilon x^{\beta^k}
 + \epsilon\left(
 x^{2^{-1}\beta^1} - \frac{a}{2\epsilon}x^{\sum_{k=2}^N 2^{-k}\beta^k + 2^{-N}\beta^{N+1}}
 \right)^2 \\
& \hspace{5ex} + \epsilon \left(x^{2^{-1}\beta^2}
 - \frac{1}{2}\left(\frac{a}{2\epsilon}\right)^2 x^{\sum_{k=3}^N 2^{-k+1}\beta^k + 2^{-N+1}\beta^{N+1}}
 \right)^2 \\
 & \hspace{5ex} - \epsilon\left(\frac{1}{2}\left(\frac{a}{2\epsilon}\right)^2\right)^2
x^{\sum_{k=3}^N 2^{-k+2}\beta^k + 2^{-N+2}\beta^{N+1}}\\
 & \hspace{5ex} \vdots\\
& = \epsilon x^{\beta^{N+1}}
 + \sum_{j=1}^{N-1}\epsilon\left(
 x^{2^{-1}\beta^j} - C_jx^{\sum_{k=j+1}^N 2^{-k+j-1}\beta^k + 2^{-N+j-1}\beta^{N+1}}
 \right)^2\\
& \hspace{5ex} + \epsilon\left(x^{2^{-1}\beta^N} - C_N x^{2^{-1}\beta^{N+1}
 }\right)^2
 - \epsilon C_N^2 x^{\beta^{N+1}}
\end{align*}
where
\[
 C_1 = \frac{a}{2\epsilon},\ C_j = 2^{-1}C_{j-1}^2,\ j= 1,2,\ldots,N.
\]
Thus we have    
\[
 C_j=\frac{a^{2^{j-1}}}{2^{2^{j}-1}\epsilon^{2^{j-1}}},\ j= 1,2,\ldots,N.
\]
    By Proposition $\ref{prop:integrity_binary}$,
    we have $\sum_{k=j+1}^N 2^{-k+j-1}\beta^k + 2^{-N+j-1}\beta^{N+1}$ is
    contained in $\Z_+^n$ for each $j$.
    Therefore $\sum_{k=1}^{N+1}\epsilon x^{\beta^k} - a x^\alpha +
    \epsilon C_N^2x^{\beta^{N+1}}\in \sum\R[x]^2$.\\
Case $t = \{2,\dots,N\}$.
 \begin{align*}
  &     \sum_{k=1}^{N+1} \epsilon x^{\beta^k} - ax^\alpha \\
  & = \sum_{j=t}^{N+1}\epsilon x^{\beta^{j}}
 + \sum_{j=1}^{t-1}\epsilon\left(
 x^{2^{-1}\beta^j} + C_jx^{\sum_{k=j+1}^N 2^{-k+j-1}\beta^k + 2^{-N+j-1}\beta^{N+1}}
  \right)^2\\
& \hspace{5ex} - \epsilon C_{t-1}^2x^{\sum_{k=t}^N 2^{-k+t-1}\beta^{k} + 2^{-N+j-1}\beta^{N+1}}\\
 & = \sum_{j=t}^{N+1}\epsilon x^{\beta^{j}}
 + \sum_{j=1}^{t-1}\epsilon\left(
 x^{2^{-1}\beta^j} + C_jx^{\sum_{k=j+1}^N 2^{-k+j-1}\beta^k + 2^{-N+j-1}\beta^{N+1}}
  \right)^2   - L x^{\beta^t}\\
  & \hspace{5ex} + L x^{\beta^t}
- \epsilon C_{t-1}^2x^{\sum_{k=t}^N 2^{-k+t-1}\beta^{k} +
  2^{-N+t-1}\beta^{N+1}}\\
  & = \sum_{j=t+1}^{N+1}\epsilon x^{\beta^{j}}
 + \sum_{j=1}^{t-1}\epsilon\left(
 x^{2^{-1}\beta^j} + C_jx^{\sum_{k=j+1}^N 2^{-k+j-1}\beta^k+ 2^{-N+j-1}\beta^N}
  \right)^2   - (L + \epsilon) x^{\beta^t}\\
  & \hspace{5ex} + L \left(x^{2^{-1}\beta^t}
- \frac{C_{t}}{2L}x^{\sum_{k=t+1}^N 2^{-k+t-1}\beta^{k} + 2^{-N+t-1}\beta^{N+1}}
\right)^2 
- \frac{C_t^2}{2^2 L}x^{\sum_{k=t+1}^N 2^{-k+t}\beta^{k} +
  2^{-N+t}\beta^{N+1}}\\
  & = \sum_{j=t+2}^{N+1}\epsilon x^{\beta^{j}}
  - (L + \epsilon) x^{\beta^t}
 + \sum_{j=1}^{t-1}\epsilon\left(
 x^{2^{-1}\beta^j} + C_jx^{\sum_{k=j+1}^N 2^{-k+j-1}\beta^k + 2^{-N+j-1}\beta^N}
  \right)^2   \\
  & \hspace{5ex} + L \left(x^{2^{-1}\beta^t}
- \frac{C_{t}}{2L}x^{\sum_{k=t+1}^N 2^{-k+t-1}\beta^{k} + 2^{-N+t-1}\beta^{N+1}}
  \right)^2 \\
  &
  \hspace{5ex} + \epsilon \left(x^{2^{-1}\beta^{t+1}}
- \frac{1}{2\epsilon}\frac{C_{t}^2}{2^2L}x^{\sum_{k=t+2}^N 2^{-k+t}\beta^{k} + 2^{-N+t}\beta^{N+1}}
  \right)^2\\
& \hspace{5ex} - \frac{1}{2^2\epsilon}\left(\frac{C_t^2}{2^2 L}\right)^2 x^{\sum_{k=t+2}^N 2^{-k+t+1}\beta^{k} +
  2^{-N+t+1}\beta^{N+1}}\\
  & = \epsilon x^{\beta^{N+1}}
  - (L + \epsilon) x^{\beta^t}\\
& \hspace{5ex} + \sum_{j=1}^{t-1}\epsilon\left(
 x^{2^{-1}\beta^j} + C_jx^{\sum_{k=j+1}^N 2^{-k+j-1}\beta^k + 2^{-N+j-1}\beta^{N+1}}
  \right)^2   \\
  & \hspace{5ex} + L \left(x^{2^{-1}\beta^t}
- \frac{C_{t}}{2L}x^{\sum_{k=t+1}^N 2^{-k+t-1}\beta^{k} + 2^{-N+t-1}\beta^{N+1}}
  \right)^2 \\
  &
 \hspace{5ex} + \sum_{j=t+1}^N\epsilon \left(x^{2^{-1}\beta^{j}}
- D_{j} x^{\sum_{k=j+1}^N 2^{-k+j-1}\beta^{k} + 2^{-N+j-1}\beta^{N+1}}
  \right)^2 - \epsilon D_{N}^2 x^{\beta^{N+1}},
 \end{align*}
    where
\[
    D_{t+1} = \frac{C_t^2}{2^3\epsilon L},\quad D_j = 2^{-1}D_{j-1}^2,\
    j = t+2,\ldots, N,
\]
    and hence we have
\[
 D_j =
    \frac{1}{2^{2^{j-t-1}-1}}\left(\frac{a^{2^t}}{2^{2^{t+1}+1}\epsilon^{2^t+1}L}\right)^{2^{j-t-1}},\
    j = t+2,\ldots, N.
\]
Then by taking $L$ large so that $D_N<1$, we obtain
    $\sum_{k=1}^{N+1}\epsilon x^{\beta^k} - ax^\alpha + (L +
    \epsilon)x^{\beta^t}\in \sum\R[x]^2$.
     Case $t = 1$ is identical to the case $t = N+1$.
   \end{proof}
  

\begin{lemma}
\label{lemma:monomial_sos}
For $f\in \R[x]$ with $f(0) = 0$, let $\gamma$ be a face of $\Gamma(f)$. 
Suppose that $f_{\gamma} \in \rint\left(\sum\R[x]^2_{\frac{1}{2}{\gamma}} \right)$.
Then for any $a>0$, $\alpha \in \bconv\Delta_E(f_\gamma)\setminus \gamma$, 
\[
 f_\gamma \pm ax^\alpha \in \sum\R[[x]]^2.
\]
\end{lemma}

\begin{proof}
 Since
 $f_\gamma\in \rint\left(\sum\R[x]^2_{\frac{1}{2}\gamma} \right)$,
there exists $\epsilon>0$ such that $f_\gamma - \epsilon p_\gamma\in \sum\R[x]^2$.
 Let arbitrary $\alpha\in \bconv\Delta_E(f_\gamma)\setminus\gamma$ be fixed.
 Then there exist $\{\beta^k\}\in \Delta_E(f_\gamma)\cap (2\Z_+)^n$ such that
 $\alpha = \sum_{k=1}^N 2^{-k}\beta^k + 2^{-N}\beta^{N+1}$.
     Since $\gamma$ is a face, there exist
   $A=(A_1,\ldots,A_n)\in \Z_+^n\setminus\{0\}^n$ and $v>0$ such that
   $\{\alpha'\in \Z_+^n\mid A\cdot\alpha'=v\}$ contains $\gamma$.
  By taking the dot product of $A$ and $\alpha$, we have
   \[
   A\cdot\alpha = \sum_{k=1}^N \frac{1}{2^k}A\cdot\beta^k +
   \frac{1}{2^N}A\cdot\beta^{N+1}.
   \]
   Since $\alpha\notin \gamma$, we have $A\cdot\alpha>v$. In addition, since
 $\sum_{k=1}^N 2^{-k} + 2^{-N-1}=1$,
   there exists $t\in \{1,\ldots,N+1\}$ such that
   $A\cdot\beta^t> v$ and thus
   $\beta^t\notin \gamma$.
%
 
Now, for $M>0$ we have
\begin{equation*}
\begin{split}
 & f_\gamma \pm a x^\alpha \\
 & = f_\gamma -\epsilon p_\gamma + \epsilon
 p_\gamma - \frac{\epsilon}{N+2} \sum_{k=1}^{N+1}x^{\beta^k}
 + \frac{\epsilon}{N+2} \sum_{k=1}^{N+1}x^{\beta^k} \pm a x^\alpha - Mx^{\beta^t} + Mx^{\beta^t}\\
 & = \left(f_\gamma - \epsilon p_\gamma\right)
 + \left(\epsilon p_\gamma - \frac{\epsilon}{N+2} \sum_{k=1}^{N+1}x^{\beta^k} - Mx^{\beta^t}\right)\\
& + \left(\frac{\epsilon}{N+2} \sum_{k=1}^{N+1}x^{\beta^k} \pm a x^\alpha + Mx^{\beta^t}\right).
\end{split}
\label{eq1}
\end{equation*}
 By Lemma $\ref{lemma:binary_sos}$, there exists $M>0$ such that the last parenthesis is contained in $\sum\R[x]^2$.
 Since $\beta^t\in \Delta_E(f_\gamma)\setminus\gamma\cap\Z^n$,
 there exist $\tilde{\beta^t}\in \gamma\cap(2\Z_+)^n$ and $\omega\in
 (2\Z_+)^n\setminus\{0\}^n$ such that
 $\beta^t = \tilde{\beta^t} + \omega$.
 In addition, let $r=\#\{\beta^k\mid
 \beta^k=\beta^t,k=1,\ldots,N+1\}$ and $\tr=\#\{\beta^k\mid \beta^k=\tilde{\beta^t},k=1,\ldots,N+1\}$.
 Then $0\leq r,\tr\leq N+1$ and
\begin{align*}
 &  \epsilon p_\gamma - \frac{\epsilon}{N+2} \sum_{k=1}^{N+1}x^{\beta^k} - Mx^{\beta^t}\\
& =\epsilon \sum_{\alpha'\in \gamma\cap
 (2\Z_+)^n\atop{\alpha'\neq \tilde{\beta^t}}}x^{\alpha'}
 + \epsilon x^{\tilde{\beta^t}}
 -\frac{\tr\epsilon}{N+2}x^{\tilde{\beta^k}}
 - \frac{\epsilon}{N+2}\sum_{k\neq
 t\atop{\beta^k\neq\beta^t,\tilde{\beta^k}}}x^{\beta^k} \\
 & \hspace{5ex} - (r\epsilon(N+2)^{-1} + M)x^{\beta^t}\\
 & =\epsilon \sum_{\alpha'\in \gamma\cap
 (2\Z_+)^n\atop{\alpha'\neq \tilde{\beta^t}}}x^{\alpha'}
 - \frac{\epsilon}{N+2}\sum_{k\neq t\atop{\beta^k\neq\beta^t,\tilde{\beta^k}}}x^{\beta^k}\\
& \hspace{5ex} + \epsilon x^{\tilde{\beta^t}}\left(1-\tr(N+2)^{-1}- \epsilon^{-1}(r\epsilon(N+2)^{-1} +
 M)x^{\omega}\right)
\end{align*}
 is contained in $\sum\R[[x]]^2$.
Therefore, we have $f_\gamma \pm a x^\alpha \in \sum\R[[x]]^2$.

\end{proof}

\begin{example}
 Let $f(x,y) = x^{16} + y^{10} - x^{13} y^2$ and $\Gamma = \Gamma(f)$.
 Then $\Gamma = \{\lambda(16,0) + (1-\lambda)(0,10)\mid 0\leq\lambda\leq1\}$ and 
 $\Delta_E(f_\Gamma) = \left\{(16,0) + (2\Z_+)^2\right\}
 \cup \left\{(0,10) + (2\Z_+)^2\right\}$. 
 We have $(13,2)\in \bconv \Delta_E$. In fact, 
\[
   (13,2) = \frac{1}{2}(16,0) + \frac{1}{2^2}(16,0)
 + \frac{1}{2^3}(0,10) + \frac{1}{2^4}(16,0) + \frac{1}{2^4}(0,12)
\]
 and $(0,12) = (0,10) + (0,2) \in \Delta_E(f_\Gamma)\setminus\Gamma$.
 Now we have
 \begin{align*}
  \frac{1}{2^2}(16,0)
 + \frac{1}{2^3}(0,10) + \frac{1}{2^4}(16,0) + \frac{1}{2^4}(0,12) & =
  (5,2)\\
  \frac{1}{2^2}(0,10) + \frac{1}{2^3}(16,0) + \frac{1}{2^3}(0,12) & =
  (2,4)\\
   \frac{1}{2^2}(16,0) + \frac{1}{2^2}(0,12) & = (4,3)\\
 \frac{1}{2}(0,12) & = (0,6).
 \end{align*}
 Thus we obtain that for any $\epsilon_0>0$ there exists $M>0$
\begin{multline*}
   \epsilon_0\left(3x^{16} + y^{10} + y^{12}\right) - x^{13} y^2
  + My^{12}\\
  = \epsilon_0(x^8 - (2^{-1}\epsilon_0^{-1})x^5y^2)^2
 + \epsilon_0(x^8 - (2^{-3}\epsilon_0^{-2})x^2y^4)^2
  + \epsilon_0(y^5 - (2^{-7}\epsilon_0^{-4})x^4y^3)^2\\
 + \epsilon_0(x^8 - (2^{-15}\epsilon_0^{-8})y^6)^2
 + (\epsilon_0 + M - 2^{-30}\epsilon_0^{-15})y^{12}.
\end{multline*}
 is contained in $\sum\R[x]^2$.
 Therefore
 \begin{align*}
  f(x,y) & = x^{16} + y^{10} - x^{13} y^2 - \epsilon(x^{16} + y^{16})
  + \epsilon(x^{16} + y^{10})\\
& \hspace{3ex} - 6^{-1}\epsilon\left(3x^{16} +  y^{10} + y^{12}\right) +
  6^{-1}\epsilon\left(3x^{16} +  y^{10} + y^{12}\right)
  - My^{12} + My^{12}
  \\
  & = (1-\epsilon)x^{16} + (1-\epsilon)y^{10} 
  + \epsilon(1-2^{-1})x^{16}\\
& \hspace{3ex} + \epsilon y^{10}\left(1-6^{-1} - \epsilon^{-1}(6^{-1}\epsilon + M)y^{2}\right) \\
& 
\hspace{3ex} +  6^{-1}\epsilon\left(3x^{16} +  y^{10} + y^{12}\right)
  - x^{13} y^2 + My^{12}  
 \end{align*}
 is contained in $\sum\R[[x]]^2$.
\end{example}
\begin{proof}
[Proof of Theorem $\ref{thm:sufficiency}$]
 For a maximal face $\gamma$ of $\Gamma:=\Gamma(f)$,
let $s$ be the number of elements of $\supp f\cap\left(\conv\Delta(f_\gamma)\setminus
 \gamma\right)$.
	For arbitrary small $\epsilon>0$ and each $\alpha \in \supp
 f\cap\left(\conv\Delta(f_\gamma)\setminus 
 \gamma\right)$, Lemma \ref{lemma:monomial_sos} ensures that
\[
  \frac{\epsilon}{s}f_\gamma + f_\alpha x^\alpha\in \sum\R[[x]]^2.
\]
Therefore 	
	\[
	 \epsilon f_\gamma + \sum_{\alpha}\{f_\alpha x^\alpha\mid 
	\alpha \in \supp f\cap\left(\Delta(f_\gamma)\setminus \gamma\right)\}
	\]
 is contained in $\sum\R[[x]]^2$.
 Let us consider the right hand side of
\[
 f = f_\Gamma - \epsilon \sum_{\gamma}f_\gamma 
+ 
\epsilon \sum_{\gamma}f_\gamma + 
\sum_{\alpha}\{f_\alpha x^\alpha\mid 
	\alpha \in \supp f\setminus \Gamma\},
\]
where the first and second summations are taken 
 with respect to every maximal face $\gamma$ of $\Gamma$.
 Since $\epsilon>0$ is an arbitrary small constant
 and $f_{\Gamma}\in \rint\left(\sum\R[x]^2_{\frac{1}{2}\Gamma}\right)$,
 we have $f_\Gamma - \epsilon \sum_{\gamma}f_\gamma \in
 \sum\R[x]^2_{\frac{1}{2}\Gamma}$ and hence $f \in \sum\R[[x]]^2$. 
\end{proof}

\subsection{Regularity of Newton polyhedra}
In Theorem $\ref{thm:sufficiency}$, Condition $(4)$ is hard to check.
However there are some kinds of Newton diagrams which the condition is automatically
satisfied.
In addition, it will be shown that
when we use Theorem $\ref{thm:sufficiency}$,
we need to check the condition for only lower degree parts of polynomials.
First we define a regularity property of Newton polyhedra.

 \begin{definition}
  \label{def:regularity}
Let $f\in \R[x]$ with $f(0)=0$.
 We say that $f$ has a \textit{regular Newton polyhedron}, if $f$
 satisfies that  
\begin{enumerate}
 \item Every vertex of $\Gamma(f)$ is even;
 \item For each vertex $\alpha$ of $\Gamma(f)$, $f_\alpha>0$;
\item If for each maximal face $\gamma$ of $\Gamma(f)$,
        \begin{multline*}
         \{\alpha \in \supp f\cap \conv\Delta(f_\gamma)\setminus\gamma\mid \alpha \text{ is odd or } f_\alpha<0\}
	 \subset \bconv \Delta_E(f_\gamma).
       \end{multline*}
\end{enumerate}

 \end{definition}
 With this regularity, Theorem $\ref{thm:sufficiency}$
 can be restated as
follows:
 \begin{thm}
  \label{thm:sufficiency_reg}
  Let $f\in \R[x]$ with $f(0) = 0$.
  Suppose that $f$ has a regular Newton polyhedron.
  If we have 
  $f_{\Gamma(f)}(x)\in \rint\left(\sum\R[x]_{\frac{1}{2}{\Gamma(f)}}^2\right)$,
  then $f\in \sum\R[[x]]^2$.
 \end{thm}

%
%
%
%
%


The following proposition explains a different aspect of Lemma $\ref{lemma:pd}$
that if a Newton diagram is included in the plane $|\alpha|=2$ and meets
all coordinate axes, its Newton polyhedron is regular.

 \begin{prop}
 \label{prop:quad}
   Let $f\in \R[x]$. Suppose that
\[
 \Gamma:=\Gamma(f) = \left\{\alpha\in \Z_+^n\mid \alpha_1+
 \cdots + \alpha_{n-1} + \alpha_n = 2 \right\}.
\]
  If $f_\Gamma$ is positive definite,
  then $\conv\Delta(f_\Gamma)\cap Z^n\subset \bconv\Delta_E(f_\Gamma)$
  and thus $f$ has a regular Newton polyhedron.
 \end{prop}
  \begin{proof}
Let $f = \sum_k f_k$ be the expansion of
its homogeneous components where $\deg f_k = k$.   
   We note that the assumption is equivalent to
that $f_0=f_1 = 0$ and $f_2$ is positive definite.
   We show the conclusion by induction on the number of variables.
   
  If $n = 1$, we can write $f = f_2x^2 + \sum_{k=3}^d f_k x^k$ where 
  $d = \deg f$.
  Then $f_2>0$ and $\supp f\subset \{2\} + \Z_+$.
  Thus $\conv\Delta(f_\Gamma)\cap \Z\subset \Delta_E(f_\Gamma)$.
  
  Suppose that the conclusion holds for $n$.
  Let $f\in \R[x_1,\ldots,x_{n+1}]$ be such that 
\[
 \Gamma = \Gamma(f) = \{\alpha \in \Z_+^{n+1} \mid \alpha_1 + \cdots +
  \alpha_{n+1} = 2\}
\]
   and $f_\Gamma$ is positive definite. 
  Then for the canonical basis $\{e_i\}$ of $\Z^{n+1}$, we have $2e_i\in \Gamma
  \cap \supp f_2$ for $i = 1, \dots, n+1$.
  Clearly, $f$ satisfies the condition $(1)$ and $(2)$ of Definition \ref{def:regularity}.
  Suppose $\alpha \in \conv\Delta(f_\Gamma)\cap \Z^{n+1}$.
  
  Case $\alpha_{n+1} \geq 2$. Then
  \[
  \alpha\in \{2e_{n+1}\} + \Z_+^{n+1}
  \subset \supp f_\Gamma\cap(2\Z)^{n+1} + \Z_+^{n+1}
  \subset \Delta_E(f_\Gamma)\cap \Z_+^{n+1}.
  \]
  Since $\Delta_E(f_\Gamma)\cap \Z_+^{n+1} \subset \bconv
  \Delta_E(f_\Gamma)$, we have $\alpha \in \bconv \Delta_E(f_\Gamma)$.
  
  Case $\alpha_{n+1} = 1$. Then $\alpha = e_{n+1} + (\beta,0)$ for some
  $\beta\in \Z_+^{n}$. Now we have
  \[
   \alpha = \frac{1}{2}\{2e_{n+1} + (2\beta,0)\}.
  \]
  Since at least one component of $2\beta$ is greater than or equal to $2$,
  the same arguments in the previous case implies that $(2\beta,0)\in \Delta_E(f_\Gamma)$. In addition $2e_{n+1}\in
  \Delta_E(f_\Gamma)$ and thus $\alpha \in \bconv\Delta_E(f_\Gamma)$.

  Case $\alpha_{n+1} = 0$.
  Then $\alpha = (\talpha,0)$ for some $\talpha\in Z_+^n$.
  Define $\tf = f(x_1,\ldots,x_n,0)$.
  Then $\tf\in \R[x_1,\ldots,x_n]$, $\tf_0 = \tf_1 = 0$ and $\tf_2$ is
  positive definite.
  Since $\{\alpha\in \supp f_2
  \mid \alpha_{n+1} = 0\} = \supp\tf_2 \times \{0\}$,
  we have
\[
    \alpha \in (\conv\Delta(\tf_{\tGamma})\cap \Z^n)\times \{0\}  \subset \bconv\Delta_E(\tf_{\tGamma})\times \{0\}, 
\]
  where the inclusion is implied by the induction hypothesis.
  Now we claim that $\Delta_E(\tf_{\tGamma})\times \{0\} \subset
  \Delta_E(f_\Gamma)$.
  Let $\alpha'\in \Delta_E(\tf_{\tGamma})\times \{0\}$. Then $\alpha' =
  (\beta + r, 0)$ for some $\beta\in \supp \tf_{\tGamma}\cap(2\Z)^n$, $r\in
  \R_+^{n}$. Since $(\beta,0)\in \supp
  f_\Gamma\cap (2\Z)^{n+1}$,
  we have
   $(\beta + r,0) = (\beta,0) + (r,0)\in \Delta_E(f_\Gamma)$.
  Thus
  \[
  \alpha \in \bconv\Delta_E(\tf_{\tGamma})\times \{0\} =
  \bconv(\Delta_E(\tf_{\tGamma})\times \{0\}) \subset \bconv\Delta_E(f_\Gamma).
  \]

  Therefore $\conv\Delta(f_\Gamma)\cap\subset\bconv\Delta_E(f_\Gamma)$.
  Since $\Gamma$ is the unique maximal face of $f$, $f$ has a regular
  Newton polyhedron.
  \end{proof}

 In the case that a Newton diagram is contained in a plane,
 we can slightly relax a condition of Theorem $\ref{thm:homogeneous}$
 which means that it has to be parallel
 to the plane $|\alpha|=2$.
\begin{thm}
\label{thm:almost_quad}
 Let $f\in \R[x]$. Suppose that
\[
 \Gamma:=\Gamma(f) = \left\{\alpha\in \Z_+^n\mid k\alpha_1+
 \cdots + k\alpha_{n-1} + \alpha_n = 2k \right\}
\]
for some $k\in \Z_+$.
 If $f_\Gamma\in \rint(\sum\R[x]_{\frac{1}{2}\Gamma}^2)$,  then
 $f$ has a regular Newton polyhedron.
\end{thm}
 \begin{proof}
%
Suppose $\alpha = (\talpha,\alpha_n)\in (\conv\Delta(f_\Gamma)\cap
  \Z_+^n)\setminus\Gamma$.
Then $k|\talpha| + \alpha_n > 2k$.

 Case $|\talpha| \geq 2$. 
  Let $\gamma=\Gamma \cap \{\alpha\in \Z_+^n\mid \alpha_n = 0\}$
  Then 
  $\gamma = \{(\alpha',0) \in \Z_+^n\mid \alpha'_1 + \cdots + \alpha'_{n-1}
  = 2\}$ and $\gamma$ is a face of $\Gamma$.
  In addition $\gamma=\Gamma(\tf)\times \{0\}$ where
  $\tf(x_1,\ldots,x_{n-1}) = f(x_1,\dots,x_{n-1},0)$.
  Let $\tgamma=\Gamma(\tf)$.
  Then $f_\gamma = \tf_{\tgamma}$ and $\talpha\in \conv\Delta(\tf_{\tgamma})\cap \Z^{n-1}$.
  By Lemma $\ref{lemma:principal}$, 
$f_\Gamma - \epsilon p_\Gamma$ belongs to $\sum\R[x]^2$ 
for a sufficiently small $\epsilon>0$.
Applying Theorem \ref{thm:necessity} to the face $\gamma$ of
$\Gamma$, we also have $f_\gamma - \epsilon p_\gamma\in \sum\R[x]^2$.
Then
  we have $\tf_{\tgamma} = f_\gamma=(f_\gamma - \epsilon p_\gamma) + \epsilon p_\gamma$ is
  a positive definite quadratic form in $x_1,\ldots, x_{n-1}$.
  Thus Proposition $\ref{prop:quad}$ implies that
\begin{multline*}
   \alpha = (\talpha,0) + (0,\ldots,0,\alpha_{n}) \in
  \conv\Delta(\tf_{\tgamma})\cap \Z^{n-1}\times\{0\} + \Z_+^n\\
  \subset \conv\Delta(\tf_{\tgamma})\cap \Z^{n-1} \times \Z_+
 \subset \bconv\Delta_E(\tf_{\tgamma})\times \Z_+.
\end{multline*}
Since $\supp \tf_{\tgamma} \cap (2\Z)^{n-1}\times \{0\}\subset \supp f_\gamma \cap
  (2\Z)^n$, we have
  $\Delta_E(\tf_{\tgamma})\times \Z_+ \subset \Delta_E(f_\gamma)$.
  Thus
  \[
   \alpha\in \bconv\Delta_E(\tf_{\tgamma})\times \Z_+\subset
  \bconv\Delta_E(f_{\gamma}) \subset \bconv\Delta_E(f_\Gamma).
  \]
  

Case $|\talpha|= 1$.
Notice that $\alpha_n \geq k$ and there exists an unique index $t$ such
 that
$\alpha_t=1$ and $\alpha_s=0$ for $s\neq t$. 
Suppose that $t = 1$. Then we have
\begin{multline*}
 \alpha = (1, 0, \cdots, 0, \alpha_n) \\
 = \frac{1}{2}\left\{(2,0, \cdots, 0) +
 (0, \cdots, 0, 2\alpha_n)\right\}\in \bconv \Delta_E(f_\Gamma).
\end{multline*}
The same argument gives the inclusion for the case $t=2,\dots,n$.
  
The case $|\talpha| = 0$ is obvious.
 \end{proof}


\begin{example}
Let $f(x,y,z) = x^2 + y^2 + xyz + yz^6 + z^{10}$.
Then $\Gamma(f) = \{\alpha\in \R_+^3\mid 5\alpha_1 + 5\alpha_2 + \alpha =
 10\}$.
Here the lowest form $g(x,y,z) = x^2 + y^2$ is only a positive
 semidefinite form and thus Lemma $\ref{lemma:pd}$ can not be applied.
 However $f\in\sum\R[[x]]^2$ by
 Theorem $\ref{thm:almost_quad}$ and Theorem $\ref{thm:sufficiency_reg}$.
In fact, we can see it directly by
\[
 f = y^2\left(1 - \frac{3}{4}z^2\right) + \left(x +
 \frac{1}{2}yz\right)^2
 + \frac{1}{2}z^{10} + \frac{1}{2}\left(yz + z^5\right)^2.	     
\]
\end{example}

The following proposition ensures that
the regularity of lower degree parts is enough for a polynomials
to belong $\sum\R[[x]]^2$.
\begin{thm}
\label{thm:bounded_reg}
 Suppose that $f\in \R[x]$ satisfies the following;
\begin{enumerate}
\item $\Gamma:=\Gamma(f)$ meets all coordinate axes;
\item $f_\Gamma(x)\in
 \rint\left(\sum\R[x]_{\frac{1}{2}\Gamma}^2\right)$.
\end{enumerate}
If $\sum\{f_\alpha x^\alpha:|\alpha|\leq \deg(f_\Gamma)+1\}$
has a regular Newton polyhedron,
then we have $f\in \sum\R[[x]]^2$.
\end{thm}
\begin{proof}
Let $d=\deg(f_\Gamma)$, $f_0 = \sum\{f_\alpha x^\alpha:|\alpha|\leq d
 + 1\}$. 
 Then there exists $\epsilon>$ such that
$f_0 - \epsilon p_\Gamma\in
 \rint\left(\sum\R[x]_{\frac{1}{2}\Gamma}^2\right)$.

Since $d$ is even, Lemma $\ref{lemma:GM}$ ensures 
that for any $K>0$
there exists $M>0$ such that
\[
 M\sum_ix^{d+2}_i + \sum_{|\alpha| = d+2} f_\alpha x^\alpha\in
 \rint\sum\R[x]_{\frac{d}{2}+1}^2.
\]
\begin{multline*}
 f = \left(f_0  - \epsilon p_\Gamma\right)
 + \left(\epsilon p_\Gamma - M\sum_ix^{d+2}_i \right)\\
 + \left(M\sum_ix^{d+2}_i + \sum_{|\alpha|= d + 2}f_\alpha x^\alpha
 + \sum_{|\alpha|> d + 2}f_\alpha x^\alpha\right)
\end{multline*}
 Since $\Gamma(f)$ meets all coordinate axes, the second parenthesis is
 contained in $\sum\R[[x]]^2$.
 By Theorem $\ref{thm:homogeneous}$, the last parenthesis is contained
 in $\sum\R[x]^2$.
\end{proof}

As an easy consequence of Theorem $\ref{thm:bounded_reg}$,
if the Newton diagram stays away from other exponents,
regularity is not necessary to ensure $f\in \sum\R[[x]]^2$.

\begin{cor}
 Suppose that $f\in \R[x]$ satisfies
\begin{enumerate}
\item $\Gamma(f)$ meets all coordinate axes;
\item $f_\Gamma(x)\in
 \rint\left(\sum\R[x]_{\frac{1}{2}\Gamma}^2\right)$.
\end{enumerate}
If 
the degree of each monomial in $f - f_\Gamma$ is greater than $\deg(f_\Gamma) + 1$, then we have $f\in \sum\R[[x]]^2$.
\end{cor}

\section{Constrained case}

In this section, we seek a sufficient condition for $f\in \R[x]$ to belong to 
a quadratic module generated by several polynomials.
Here we consider a \textit{local order} on monomials in $\R[x]$.
For example, the \textit{anti-graded rex order} on $\R[x,y]$ is a local
order
satisfying that
\[
 1 > x > y > x^2 > xy > y^2.
\]
For the detailed definition and discussion, see \cite[Section 4.3]{CLO}.
For a given ordering, the \textit{leading term} $\LT(f)$ of $f$ 
be the maximal monomial appearing in $f$.
The following theorem is well-known \cite[Cor. 3.13 in Chap.4]{CLO}.



\begin{thm}
[Mora's division]
  For $f, g_i\in \R[x], i = 1, \ldots, l$ and a local order $>$,
 there exist $u, q_i, r\in \R[x]$ such that
 \begin{enumerate}
  \item $(1+u)f = \sum_i q_ig_i + r$,
 \item $u(0) = 0$,
 \item $\LT(f)\geq \LT(q_i g_i)$ for all $i$,
 \item $\LT(r)$ can not be divided by $\LT(g_i)$ for all $i$.
 \end{enumerate}
\end{thm}

Here we consider slightly modified version of the division.
\begin{definition}
[Modified Mora's division]
 After applying the Mora's division
\[
 (1+u)f = \sum_i q_ig_i + r,
\]
let $r_0$ be the polynomial obtained by eliminating 
all terms of $r$ included in 
the ideal generated by the leading monomials of linear parts of $g_i, h_j$.
For $\Gamma:= \Gamma(r_0)$, let $d = \deg(r_{0,\Gamma})$.
\begin{enumerate}
\item Divide further as
{\small \[
 (1+u')f = \sum_i q'_ig_i + r',
 \] }
 where any monomials of $r'$ with the degree $\leq d + 1$ can not be divided by
      $\LT(g_i)$ for all $i$.
\item Let $\tr$ be a polynomial obtained by eliminating all monomials of
      $r'$ with degree $> d +1$.
\end{enumerate}

We call $\br$ the \textit{essential remainder}.
\end{definition}

For $f\in \R[x]$, we use the notation $f_z(x) := f(x+z) - f(z)$.
Note that $f_z(0) = 0$.
For $g_i\in \R[x], i = 1, \ldots, l$, let $\langle g_1, \ldots,
g_l\rangle^{\sim} = \{\sum_i \tau_ig_i \mid \tau_i \in \R[[x]]\}$.

\begin{thm}
\label{thm:constrained}
For a global minimizer $z$ of $\mathrm{(POP)}$, let $L = f - \sum_{i=1}^l \lambda_ig_{i} - \sum_{j=1}^m
	\mu_jh_{j}$ with $\lambda_i\geq0 , \mu_j\in \R$ satisfying
 $\nabla L(z) = 0$ and $\lambda_i g_i(z) = 0$. 
Suppose that  for a local order, 
an essential remainder $\tr$ of modified Mora's division 
of
\[
L_z \text{ by } \{ \lambda_i g_{i,z},h_{j,z}\ \big|\ \lambda_i\nabla
 g_i(z)\neq 0\}.
\]
satisfies the following:
\begin{enumerate}
 \item $\Gamma = \Gamma(\tr)$ meets all coordinate axes of appearing
       variables in $\tr$.
 \item $\tr_\Gamma\in \rint\sum\R[x]_{\frac{1}{2}\Gamma}^2$
 \item $\tr$ has a regular Newton polyhedron.
\end{enumerate}
Then we have $f \in \tM(g_{1,z},\ldots,g_{l,z}) + \langle h_{1,z},\ldots,h_{m,z}\rangle^\sim$.
\end{thm}

\begin{proof}
For a global minimizer $z$, let $I = \{i\mid \lambda_i\nabla g_i(z)\neq 0\}$. 
By the modified Mora's division, there exist $u, p_i, q_j, \tr, w \in \R[x]$ such that $u(0) =
 0$ and 
\[
(1+u) L_z = \sum_{i\in I} p_i\lambda_ig_{i,z} + \sum_{j=1}^m
 q_j h_{j,z}
 + \tr + w,
\]
where
$\LT(L_z)\geq \LT(p_i\lambda_i g_{i,z}), \LT(q_i h_{j,z})$ in the local order, 
each monomial of $\tr$ can not be divided by $\LT(g_{i,z}), \LT(h_{j,z})$ and $w \in
 \langle g_{i,z}, h_{j,z}\rangle_{i,j}$ and the least degree of $w\geq
 d+2$, where $d$ is the number given in
 the definition of the modified Mora's division.
Since $L(z) = 0, \nabla L(z) = 0$, we have $\deg(\LT(L_z))\geq 2$.
Then
the least degree of the monomials
 of $p_i\lambda g_{i,z}\geq 2$ for all $i\in I$.
Thus the least degree of monomials in $p_i\geq 1$
and hence $p_i(0)=0$ for all $i\in I$. 

Further by the Division theorem in $\R[[x]]$ \cite[Theorem 6.4.1]{GP}, 
there exist
$p', q', r'\in \R[[x]]$ such that
\[
 w = \sum_{i\in I} p'_i\lambda_i g_{i,z} + \sum_j q'_j h_{j,z} + r',
\]
where each monomial of $r'$ can not be divided by $\LT(\lambda_i g_i), \LT(h_j)$ and
 the least degree of $r' \geq d+2$.
Similarly, we have $p'_i(0) = 0$.
Then
\begin{align*}
  f_z & = \sum_{i=1}^l \lambda_ig_{i,z} + \sum_{j=1}^m
	\mu_jh_{j,z} + L_z\\
& = \sum_{i\in I} \lambda_i\left(1 + \frac{p_i + p'_i}{1 + u}\right)g_{i,z} +
 \sum_{j=1}^m 
	\left( \mu_j + \frac{q_j+ q'_j}{1 + u}\right)h_{j,z} \\
& \phantom{\sum_{i\in I} \lambda_i\left(1 + \frac{p_i}{1 + u}\right)}
+ \sum_{i\notin I}\lambda_ig_{i,z} 
+ \frac{\tr + r'}{1 + u}.
\end{align*}
 Since $\tr + r'$ is contained in $\sum\R[x]^2$
 by Theorem \ref{thm:bounded_reg},
 we have $f_z\in \tM(g_{1,z},\ldots,g_{l,z}) + \langle
 h_{1,z},\ldots,h_{m,z}\rangle^\sim$.

\end{proof}

\begin{example}
\begin{align*}
\min &\ f = x^3 + y^3 + z^2 + w^4 + 2\\
\text{s.t. }&\ g = 2 - x^4 - y^4 - z^4 - w^4 \geq 0
\end{align*}
The optimal is $a = (-1,-1,0,0)$. We have 
\[
\nabla f(a) = \frac{3}{4}\nabla g(a),\quad
 \nabla^2\left(f - \frac{3}{4}g\right)(a) 
= \begin{bmatrix}
   3 & 0 & 0 & 0 \\
   0 & 3 & 0 & 0 \\
   0 & 0 & 2 & 0 \\
   0 & 0 & 0 & 0
  \end{bmatrix}.
\]
Thus $\nabla^2\left(f - \frac{3}{4}g\right)(a)$ is
not positive definite on the subspace
\[
 \nabla g(a)^\perp 
= \left\langle\begin{bmatrix}
   4\\
   4\\
   0\\
   0
  \end{bmatrix}\right\rangle^\perp
= \left\langle 
\begin{bmatrix}
 1\\
 -1\\
 0 \\
 0
\end{bmatrix}, 
\begin{bmatrix}
 0 \\
 0\\
 1\\
 0
\end{bmatrix}, 
\begin{bmatrix}
 0\\
 0\\
 0\\
 1
\end{bmatrix}
\right\rangle,
\]
and hence the second order condition is not satisfied.
Let $>$ be the anti-graded rex order. We have
\begin{align*}
 & f_a  = 3x+3y-3x^{2}-3y^{2}+z^{2}+x^{3}+y^{3}+w^{4}\\
 & g_a  = 4x+4y-6x^{2}-6y^{2}+4x^{3}+4y^{3}-x^{4}-y^{4}-z^{4}-w^{4},
\end{align*}
and the remainder of $f_a$ by $g_a$ is
\begin{multline*}
 r =
 3y^{2}+z^{2}+\frac{1}{4}x^{3}-\frac{9}{4}x^{2}y+\frac{9}{4}xy^{2}-\frac{17}{4}y^{3}\\
 -\frac{3}{4}x^{4}+\frac{3}{2}x^{3}y-\frac{3}{2}xy^{3}+\frac{9}{4}y^{4}+\frac{3}{4}z^{4}
+\frac{7}{4}w^{4}+\frac{3}{8}x^{5}-\frac{3}{8}x^{4}y\\
 +\frac{3}{8}xy^{4}+\frac{3}{8}xz^{4}+\frac{3}{8}xw^{4}-\frac{3}{8}y^{5}-\frac{3}{8}yz^{4}-\frac{3}{8}yw^{4}.
\end{multline*}
By eliminating terms of $r$ contained in $\LT\langle g_a\rangle=\langle
 x \rangle$, we obtain
\[
r_0  =
 3y^{2}+z^{2}-\frac{17}{4}y^{3}+\frac{9}{4}y^{4}+\frac{3}{4}z^{4}+\frac{7}{4}w^{4}-\frac{3}{8}y^{5}-\frac{3}{8}yz^{4}-\frac{3}{8}yw^{4}.
 \]
For $\Gamma:=\Gamma(r_0)$, 
\[
 r_{0,\Gamma} = 3y^2 + z^2 + \frac{7}{4}w^4
\]
and $\deg r_{0,\Gamma} = 4$, 
Then the essential remainder $\widehat{r} = r_0$
and $\widehat{r}_{\Gamma} = r_{0,\Gamma} \in
 \rint\sum\R[x,y,z,w]_{\frac{1}{2}\Gamma}^2$.
Since the Newton diagram of $\tr$ satisfies the conditions of 
Theorem \ref{thm:almost_quad},
$\tr$ has a regular Newton polyhedron.
By Theorem \ref{thm:constrained}, we have $f\in \tM(g)$. 
\end{example}

\end{document}